\documentclass[a4paper,11pt]{article}
\usepackage[utf8]{inputenc}

\usepackage{mathtools}
\usepackage{amsmath}
\usepackage{amsthm}
\usepackage{amssymb}
\usepackage{amsfonts}
\usepackage{enumerate}
\usepackage{tikz-cd}
\usetikzlibrary{decorations.pathmorphing}
\tikzset{snake it/.style={decorate, decoration=snake}}
\usetikzlibrary{arrows}
\usepackage{mathrsfs}
\usepackage{caption}

\def\Z{\mathcal{Z}}
\def\P{\mathbb{P}}

\def\h{\mathfrak{h}}
\def\Y{\mathcal{Y}}
\def\Z{\mathcal{Z}}
\def\UC{\mathscr{U\mspace{-2mu}C}}
\def\D{\mathbb{D}}

\def\hPNG{h^{\text{PNG}}}
\def\hL{h^{(L)}}

\DeclarePairedDelimiter\floor{\lfloor}{\rfloor}
  
\newtheorem{theorem}{Theorem}

\newtheorem{proposition}[theorem]{Proposition}

\theoremstyle{definition}

\usepackage[nottoc,notlot,notlof]{tocbibind}

\title{Pushing, blocking and polynuclear growth}
\author{Will FitzGerald\thanks{Department of Mathematics, University of Manchester, Oxford Road,  
Manchester, M13 9PL, UK.
Email: william.fitzgerald@manchester.ac.uk}
}

\begin{document}
\maketitle

\abstract{We consider a discrete-time model for random interface growth which admits exact formulas and converges to the Polynuclear growth model in a particular limit. 
The height of the interface is initially flat and the evolution involves the addition of islands of height one according to a Poisson point process of nucleation events. The boundaries of these islands then spread in a stochastic manner, rather than at deterministic speed as in the Polynuclear growth model. The one-point distribution and multi-time distributions agree with point-to-line last passage percolation times in a geometric environment.
An alternative interpretation for the growth model can be given through interacting particle systems experiencing pushing and blocking interactions. }

\section{Introduction}
The Polynuclear growth model (PNG) is an extensively studied model for random interface growth in $(1+1)$-dimensions. 
The height of the interface at time $t$ is denoted by $x \rightarrow \hPNG_t(x).$ The interface grows according to nucleation events at which an island of height one is added to the interface. The left and right boundaries of these islands then spread at deterministic unit speed to the left and right respectively, and the islands merge on contact. 
All the randomness in the process is contained within the nucleation events
which occur according to a Poisson point process.  

The choice of initial condition plays an important role in the analysis of PNG. The simplest case to study is the narrow-wedge or droplet initial condition. This can be defined by setting $\hPNG_0(x) \equiv 0$ and specifying that the nucleation events occur according to a Poisson point process in the following subset of the space-time plane $\{(t, x) : \lvert x \rvert \leq t, t \geq 0\} \subset \mathbb{R}_{\geq 0} \times \mathbb{R}$. An interpretation is that there is an island spreading at unit speed from the origin at time $0$ and nucleations are only permitted on top of this island. PNG started from the droplet initial condition is connected to a Poissonised version of the longest increasing subsequence in a random permutation. 
This discovery led to exact-formulas and an understanding of the one-point \cite{BDJ} and multi-point distributions \cite{PS}.

The initial condition that is most relevant to our case is flat initial data. 
We set $\hPNG_0 \equiv 0$ and specify instead that nucleation events occur according to a Poisson point process of rate $2$ in the whole space-time plane $\mathbb{R}_{\geq 0} \times \mathbb{R}.$ We refer to this as \emph{flat PNG}. The height function in flat PNG is given by a point-to-line last passage percolation time in a Poissonian environment. Point-to-line last passage percolation 
can be related to point-to-point last passage percolation in a symmetric environment. 
As a result, flat PNG retains a connection
to a Poissonised version of a longest increasing subsequence in a random permutation but with an additional symmetry imposed and this
led to an understanding of the one-point distribution
in \cite{BR2, BR1}. This was extended to multi-point distributions in \cite{BFS} by identifying flat PNG as a particular limit related to 
discrete-time TASEP started from a periodic initial condition.
More recently exact formulas have been found for PNG with general initial data that have uncovered more general connections to Toda lattice equations \cite{MQR}. 

One aspect that is less well developed for PNG is how it can be viewed as one representative of a larger class of exactly solvable growth models. The other prototypical exactly solvable growth model, the height function of TASEP, is one example of a much larger class given by the height functions of ASEP, $q$-TASEP, PushASEP and many other interacting particle systems. 
The interpretation of PNG through last passage percolation in a Poisson environment can be viewed as an example within a larger class of exactly solvable models for directed polymers in random environment. A zero temperature example, last passage percolation with geometric data, can be viewed as a growth model called discrete PNG. 
However, the positive temperature models such as the log gamma polymer no longer retain a direct interpretation
as growth models. 
A different type of solvable deformation of PNG was recently introduced in \cite{ABW} where two islands that merge generate an island of height one at the merging location with probability $t$.

The main result of this article is to construct a new type of random interface growth model that converges to flat PNG in a particular limit and
before taking the limit preserves some of the exact solvability of flat PNG. 
In particular, reformulating the results of \cite{F} give that the one-point and multi-time distributions of this growth model agree with point-to-line last passage percolation in a geometric environment. 

We give three different descriptions of this growth process. 
The height of the interface at time $t$ is denoted by $x \rightarrow h_t(x)$ and starts from $h_0 \equiv 0$.
The description that is closest to PNG is 
that the height of the interface grows by the addition of islands of height $1$ according to a Poisson point process of nucleation events (on a discretised space-time plane) but with the difference that the boundaries of these islands now spread in a stochastic manner. The idea can be seen in Figure \ref{fig_height_function} and the definition is given in Section \ref{growth_defn2}.

A second description of the growth processes involves pushing and blocking interactions. 
A brief informal description is given here and a detailed description in Section \ref{growth_defn}. 
Given a height function 
$x \rightarrow h_t(x)$
we define the interface
$x \rightarrow h_{t+1}(x)$ by viewing the space variable $x$ in the growth process as a time variable for a Markov process. Then $h_{t+1}$ is given by the path of a continuous-time random walk jumping up with rate $v$ and down with rate $v^{-1}$ while being pushed and blocked by the path $h_t$ in order to maintain $h_{t+1}(x) \geq h_t(x)$ for all $x \in \mathbb{R}$. At this point we could simply repeat this procedure to define $h_{t+2}$ and so on. 
However, the resulting interface growth process would then lack spatial symmetry. 
We instead define $h_{t+2}$  by repeating the same process started from $h_{t+1}$ but
now viewing $(-x)$ as the time variable for a Markov process. 
It is a non-trivial fact that this does recover spatial symmetry. 

The third and final description of the growth process considered in this paper is as a particular marginal of a multi-dimensional interacting process involving pushing and blocking that was constructed in \cite{F}. This type of process was first 
considered in a Brownian setting \cite{FW} where the motivation was that marginals of a multi-dimensional process could be used to 
prove several distributional identities, e.g.~relating the point-to-line partitions function of the O'Connell Yor polymer and the log gamma polymer introduced by Sepp\"{a}l\"{a}inen. 
This provides a different context in which to consider the multi-dimensional interacting processes in \cite{F, FW}. The (tautological) identity that the height function in flat PNG is given by point-to-line last passage percolation in a Poissonian environment can be generalised in a number of ways. On the one hand, the Poissonian environment can be replaced by 
geometric/exponential environment or, in a positive temperature setting, by partition functions of the log gamma polymer. On the other hand, flat PNG 
can be replaced by other types of growth process involving push-block dynamics or reflecting Brownian motions.  
The multi-dimensional interacting processes in \cite{F, FW} then encode distributional identities between these 
different models. This is explained for one representative example in Section \ref{sec:lpp_identities}.  

In Section \ref{results} we define the growth model considered in this paper and state our main results. 
In Section \ref{growth_defn2} we interpret the points of increase and decrease of the height function as an interacting particle system. In Section \ref{sec:proof_thm} we use this viewpoint to prove convergence to PNG. In Section \ref{sec:lpp_identities} we explain identities between the growth model considered in this paper and point-to-line last passage percolation times. 

\paragraph{Acknowledgements.}
 The author is supported as part of a Leverhulme Trust Research Project Grant RPG-2021-105.

\section{Statement of Results}
\label{results}

\subsection{Definition of the growth model}
\label{growth_defn}

Define a discrete-time Markov process $(h_t)_{t \in \mathbb{Z}_{\geq 0}}$ taking values in the state space of upper semi-continuous functions from $\mathbb{R}$ into
$\mathbb{Z}_{\geq 0}$ denoted by $\UC$. The hypograph of a function $f$ is given by 
$\text{hypo}(f) = \{(x, y) : y \leq f(x)\}$. We equip $\UC$ with the topology of local Hausdorff convergence of hypographs. 
This is the topology considered in \cite{MQR} with the mild simplification that our height functions take values in $\mathbb{Z}_{\geq 0}$. 
The process starts from $h_0 \equiv 0$. 

Let $L > 0$ and define first an approximation $(\hL_t)_{t \in \mathbb{Z}_{\geq 0}}$ of $(h_t)_{t \in \mathbb{Z}_{\geq 0}}$. 
For $t \geq 1$ given $\hL_{t - 1}$ we construct $\hL_t$ in the following way which depends on whether $t$ is even or odd. 
If $t$ is odd, then consider the c\`{a}dl\`{a}g modification of $\hL_{t -1}$. 
Define $\xi(u) = \hL_{t-1}(u)$ for all $u \leq -L$ and a continuous-time Markov process $(\xi(u))_{u \geq -L}$ started
from $\xi(-L) = \hL_{t-1}(-L)$ taking values in $\mathbb{Z}_{\geq 0}$. The dynamics are not homogeneous in time
and are given as follows.
\begin{itemize}
\item $\xi$ increases by $1$ at rate $v$ whenever $u \in [-L, L]$,
\item $\xi$ decreases by $1$ at rate $v^{-1}$ whenever $\xi(u) > \hL_{t - 1}(u)$,
\item If $\hL_{t - 1}(u) = \hL_{t - 1}(u_{-}) + 1$ and $\hL_{t - 1}(u_{-}) = \xi(u_{-})$ then $\xi(u) = \xi(u_{-}) + 1$. 
\end{itemize}
We then define $h_t$ to be the upper semi-continuous modification of $\xi$. 

If $t$ is even, we repeat the same process but with time reversed. We consider the c\`{a}dl\`{a}g modification of 
$(h_{t - 1}(-u))_{u \in \mathbb{R}}$. Then define $\xi(u) = h_{t-1}(-u)$ for all $u \leq -L$ and a continuous-time Markov process $(\xi(u))_{u \geq -L}$ started
from $\xi(-L) = \hL_{t-1}(L)$ taking values in $\mathbb{Z}_{\geq 0}$. 
The dynamics are that:
\begin{itemize}
\item $\xi$ increases by $1$ at rate $v$ whenever $u \in [-L, L]$,
\item  $\xi$ decreases by $1$ at rate $v^{-1}$ whenever $\xi(u) > \hL_{t - 1}(-u)$,
\item If $\hL_{t - 1}(-u) = \hL_{t - 1}((-u)_{-}) + 1$ and $\hL_{t - 1}((-u)_{-}) = \xi(u_{-})$ then $\xi(u) = \xi(u_{-}) + 1$. 
\end{itemize}
Informally, $(\xi(u))_{u \geq -L}$ is a continuous-time random walk that experiences pushing and blocking interactions to maintain that $\xi(u) \geq \hL_{t - 1}(-u)$ for all $u \in \mathbb{R}$. 
We then set $\hL_t$ to be the upper semi-continuous modification of $(\xi(-u))_{u \in \mathbb{R}}$.

For any $n \geq 1$, as $L \rightarrow \infty$ then $(\hL_t)_{t \in \{0, 1, \ldots, n\}}$ converges weakly to a limit
$(h_t)_{t \in \{0, 1, \ldots, n\}}$ as processes on $\UC$. The convergence to a limit in $L$ is easiest to see from the formulation of the growth process in Section \ref{growth_defn2}. It is also possible to define the dynamics above directly in a stationary regime, see Section \ref{sec:lpp_identities}. 

We will use $h^{v}_t$ 
to specify the dependency on $v$.

\subsection{Last passage percolation}
We consider point-to-line last passage percolation in a geometric environment. 
Let $(g_{ij} : i, j \in \mathbb{Z}_{\geq 1}, i+j \leq 2n+1)$ be an independent collection of geometric random variables with parameter $1-v^2$.
(We use the convention that $\P(g_{ij} = k)) = (1-v^2)v^{2k}$ for all $k \geq 0$.)  
Let $\Pi_n^{\text{flat}}(k, l)$ denote the set of all directed up-right nearest neighbour paths from the point 
$(k, l)$ to the line $\{(i, j) : i + j = 2n +1\}$. 
For all $k, l \in \mathbb{Z}_{\geq 1}$ with $k + l \leq 2n+1$ let 
\[
G(k, l) = \max_{\pi \in \Pi_n^{\text{flat}}} \sum_{(i, j) \in \pi} g_{ij}.
\]
The family $(G(k, l) : k,  l \in \mathbb{Z}_{\geq 1}, k+l \leq 2n+1)$ gives a collection of point-to-line last passage percolation times where the line is fixed and the point $(k, l)$ from which the paths originate varies. 

\subsection{Main result}

The processes $(h^{v}_{\floor{2n t}})_{t \in [0, 1]}$ and $(h^{\text{PNG}}_t)_{t \in [0, 1]}$ are
c\`{a}dl\`{a}g processes with the Skorokhod topology and the underlying state space given by $\UC$. 
We denote the state space of the processes by $\D$.

\begin{theorem}
\label{main_thm}
\begin{enumerate}[(i)]
\item
Let $v_n = 1/n$, then as $n \rightarrow \infty$
\[
(h_{\floor{2n t}}^{v_n})_{t \in [0, 1]} \rightarrow (h^{\text{PNG}}_t)_{t \in [0, 1]}
\]
in the sense of weak convergence on $\D$.
\item For any $v > 0$ and any integer $n \geq 1$,
\begin{align*}
& (h^{v}(1, 0), h^{v}(2, 0), \ldots, h^{v}(2n, 0)) \\
& \stackrel{d}{=} 
(G(n+1, n), G(n, n), G(n, n-1), \ldots, G(1, 1)) .
\end{align*} 
\end{enumerate}
\end{theorem}

\section{Alternative representation of the growth process}
\label{growth_defn2}

A useful method for studying PNG, eg.~\cite{Ferrari}, is to encode the height interface in terms of two types of particles which may be referred to as points of increase/decrease, up/down-steps or antikinks/kinks.
The first step in proving Theorem \ref{main_thm} is a similar reformulation of the growth process described in Section \ref{growth_defn}.
The evolution of the height interface then corresponds to an evolution of an associated particle system. 
The advantage is that it will be possible to deduce the convergence in Theorem \ref{main_thm} from convergence of the particle system.

We define a discrete-time particle system $((\Y_t, \Z_t) : t \in \mathbb{Z}_{\geq 0})$
formed of two types of particles. We list the positions of surviving particles at time $t$ as $\Y_t = (\Y_t^1, \Y_t^2, \ldots)$ ordered as $\Y_t^1 < \Y_t^2 < \ldots$ (resp. $\Z_t = (\Z_t^1, \Z_t^2, \ldots)$ ordered as 
$\Z_t^1 < \Z_t^2 < \ldots$).  At a fixed time, the state space for each particle system is the set of locally finite 
point measures with the topology of vague convergence. 
The particles in $\Y_t$ will correspond to points of increase in a height interface and the particles in $\Z_t$ will correspond to points of decrease.   

Let $(\mathcal{N}_{t} : t = 1, 2, \ldots, 2n)$ be an independent collection of 
Poisson point processes of rate $v$ on $\{t \} \times [-L, L]$.
Let $(\zeta_t^j : j \in \mathbb{Z}_{\geq 1}, t = 1, 2, \ldots, 2n)$ be an independent collection of exponential random variables with rate $v^{-1}$. We will take $L \rightarrow \infty$ at the end of the construction.

The evolution of particle positions from time $t - 1$ to $t$ depends on whether $t$ is even or odd. 
If $t$ is odd we apply the following steps: 
\begin{itemize}
\item Nucleation points occur according to $\mathcal{N}_t$.  A nucleation point at $(t, x)$ involves the addition of a 
point of increase at position $x_{-}$ at time $t_+$ and an addition of a point of decrease at position $x_+$ at time $t_{-}$. In words the point of decrease is immediately added while the point of increase is only added once the following process has completed. 
\item The particles in $\Y$ are fixed while the particles in $\Z$ make jumps to the right. The evolution of particles proceeds sequentially from the left. Let $k_1 = \inf\{l \geq 1: \Y_t^l > \Z_{t-1}^1\}$ with the convention that the infimum is infinity if taken over the empty set and $\Y^{\infty}:=\infty$. 
If \[
\Z_{t- 1}^1 + \zeta_{t}^1 \geq \Y_t^{k_1}\]
 then there is an instantaneous pairwise annihilation and the particles
$\Z^1$ and $\Y^{k_1}$ are removed. 
Otherwise we set
\[\Z_{t}^1 = \Z_{t- 1}^1 + \zeta_{t}^1. 
\]
For $j \geq 2$ let $k_j = \inf\{l \geq 1: \Y_t^l > \Z_{t-1}^j\}$. Before taking a jump $\Z^j$ is pushed 
to $\max(\Z_{t- 1}^j, \Z_{t}^{j-1})$. 
If \[
\max(\Z_{t- 1}^j, \Z_{t}^{j-1}) + \zeta_{t}^j \geq \Y_t^{k_j}
\] then there is an instantaneous pairwise annihilation of the particles $\Z^j$ and $\Y^{k_j}$. 
Otherwise we set
\[
\Z_t^j = \max(\Z_{t- 1}^j, \Z_{t}^{j-1}) + \zeta_t^j.
\]
\end{itemize}
If $t$ is even then the evolution is the analogue of the above with the particles in $\Z$ fixed and the particle in $\Y$ making jumps to the left. Suppose that there are $m$ particles of type $1$ and $m$ particle of type $2$. 
\begin{itemize}
\item Nucleation points occur according to $\mathcal{N}_t$. A nucleation point at $(t, x)$ involves the addition of a point of increase at position $x_{-}$ at time $t_{-}$ and an addition of a point of decrease in position $x_+$ in time $t_{+}$. 
\item The evolution of particles occurs sequentially in decreasing order $\Y_t^m, \ldots, \Y_t^1$. Let $k_m = \sup\{l \geq 1: \Z_{t-1}^l < \Y_t^m\}$ with the convention that the supremum is negative infinity if taken over the empty set and 
$\Z^{-\infty}:=-\infty$. 
If \[
\Y_{t- 1}^{m} - \zeta_{t}^m \leq \Z_t^{k_m}\] then there is an instantaneous pairwise annihilation of the particles $\Y^m$ and $\Z^{k_m}$. 
Otherwise we set
\[\Y_{t}^m = \Y_{t- 1}^m - \zeta_{t}^m. 
\]
For $1 \leq j \leq m-1$ let $k_j = \sup\{l \geq 1: \Z_{t-1}^l < \Y_t^j\}$. Before taking a jump $\Y_t^j$ is pushed 
to $\min(\Y_{t- 1}^{j}, \Y_{t}^{j+1})$. 
If \[
\min(\Y_{t- 1}^{j}, \Y_{t}^{j+1}) - \zeta_{t}^j \leq \Z_t^{k_j}\]
 then there is an instantaneous pairwise annihilation of the particles $\Y^j$ and $\Z^{k_j}$. 
Otherwise we set
\[
\Y_t^j = \min(\Y_{t- 1}^{j}, \Y_{t}^{j+1}) - \zeta_{t}^j.
\]
\end{itemize}

This particle system can now be used to define a height function. 
We draw the paths of the process $(\Y_{t}, \Z_{t})_{t \in [0, 2n]}$ in the space-time plane, see Figure \ref{fig_height_function}. 
Define a height function $(\h_{t}^{(L)}(x))_{t \in [0, 2n], x \in \mathbb{R}}$ to be the number of paths of the process $(\Y_t, \Z_t)_{t \in [0, 2n]}$ that are crossed by the straight line from $(t, x)$ to $(0, x)$ as shown in Figure \ref{fig_height_function}.
As $L \rightarrow \infty$,
\[
(\h_t^{(L)}(x))_{t \in [0, 2n], x \in \mathbb{R}} \rightarrow (\h_t(x))_{t \in [0, 2n], x \in \mathbb{R}}
\]
weakly as processes on $\UC$. 

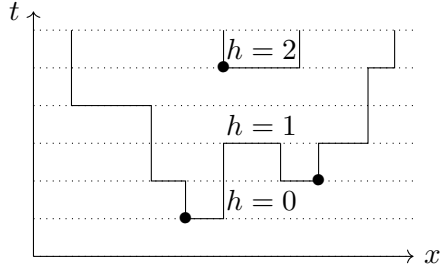
\begin{figure}
\centering
\begin{tikzpicture}[scale = 0.5]
\draw[->](0, 0) -- (0, 6.5);
\draw[->] (0, 0) -- (10, 0);
\node at (10.5, 0){$x$};
\node at (-0.5, 6.5){$t$};
\draw[dotted] (0, 0) -- (10, 0); 
\draw[dotted] (0, 1) -- (10, 1); 
\draw[dotted] (0, 2) -- (10, 2); 
\draw[dotted] (0, 3) -- (10, 3); 
\draw[dotted] (0, 4) -- (10, 4); 
\draw[dotted] (0, 5) -- (10, 5); 
\draw[dotted] (0, 6) -- (10, 6); 
\node at (4, 1){\textbullet};
\draw(4, 1) -- (5, 1)--(5, 3) -- (6.5, 3);
\draw (4, 1) -- (4, 2) -- (3.1, 2) -- (3.1, 4) -- (1, 4) -- (1, 6); 
\node at (7.5, 2){\textbullet};
\draw (7.5, 2) -- (6.5, 2) -- (6.5, 3);
\draw(7.5, 2) -- (7.5, 3) -- (8.8, 3) -- (8.8, 5) -- (9.5, 5) -- (9.5, 6);
\node at (5, 5){\textbullet};
\draw (5, 5) -- (7, 5) -- (7, 6);
\draw(5, 5) -- (5, 6);
\node at (6, 1.5) {$h = 0$};
\node at (6, 3.5){$h = 1$};
\node at (6, 5.5){$h = 2$};
\end{tikzpicture}
\caption{Defining the height function from the paths of $(\Y_t, \Z_t)_{t \in [0, 2n]}$.}
\label{fig_height_function}
\end{figure}

\begin{proposition}
$(h_t(x))_{t \in [0, 2n], x \in \mathbb{R}} \stackrel{d}{=} (\h_t(x))_{t \in [0, 2n], x \in \mathbb{R}}.$
\end{proposition}

\begin{proof}
Given a c\`{a}dl\`{a}g modification of $\hL_t$ we define 
\begin{align*}
Y_t^1 & = \inf\{x \in \mathbb{R}: \hL_t(x) = \hL_t(x_{-}) + 1\} \\
Y_t^j & =  \inf\{x > Y_t^{j-1}: \hL_t(x) = \hL_t(x_{-}) + 1\}, \qquad j \geq 2.
\end{align*}
We continue defining particle positions $(Y_t^1, \ldots, Y_t^m)$ until the infimum is taken over the empty set.
We define
\begin{align*}
Z_t^1 & = \inf\{x \in \mathbb{R}: \hL_t(x) = \hL_t(x_{-}) - 1\} \\
Z_t^j & =  \inf\{x > Z_t^{j-1}: \hL_t(x) = \hL_t(x_{-}) - 1\}, \qquad j \geq 2.
\end{align*}
The processes $(Y_t^1, \ldots, Y_t^m)$ and $(Z_t^1, \ldots, Z_t^m)$ 
give the points of increase and decrease in the height function $\hL_t$.
We show that the evolution of $(\Y_t^1, \ldots, \Y_t^m)$ and $(\Z_t^1, \ldots, \Z_t^m)$ coincides with the evolution of 
$(Y_t^1, \ldots, Y_t^m)$ and $(Z_t^1, \ldots, Z_t^m)$.

Suppose we have a height function $\hL_{t-1}$. 
Then $\hL_t$ is constructed from $\hL_{t-1}$ 
by reflecting a continuous-time random walk $\xi$ from the lower barrier $\hL_{t-1}$
as described in Section \ref{growth_defn}.
Suppose that $t$ is odd.
We consider the points of increase in the process $\xi$ as $(U_{t}^1, \ldots, U_{t}^p)$. 
This will necessitate $p$ extra points of potential decrease in the process $\xi$ which we view as initially having positions 
$(D_{t-}^1 = U_{t}^1, \ldots, D_{t-}^p = U_{t}^p)$ but their locations will be updated in the following. 
We now proceed to consider 
\[
(Y_{t-1}^1, \ldots, Y_{t-1}^m, Z_{t-1}^1, \ldots, Z_{t-1}^m, D_{t-}^1, \ldots, D_{t-}^p)
\] in increasing order. 
The points of increase $(Y_{t-1}^1, \ldots, Y_{t-1}^m)$ remain in the same position
except that there is a possibility that they could be annihilated by the motion of the points of decrease. 
The process $\xi$ does not immediately decreases when $h_{t-1}$ decreases or when $\xi$ has increased. 
Instead exponential clocks for the process $\xi$ to perform a down jump begin at
 the positions $(Z_{t-1}^1, \ldots, Z_{t-1}^m, D_{t-}^1, \ldots, D_{t-}^p).$
 If $\xi$ has not decreased before the next point of increase in $h_{t-1}$ then the two associated positions are removed from $Y$ and $(Z, D)$ respectively. This corresponds to the the pairwise annihilation in $(\Y, \Z).$ 
If an exponential clock that begins at some position say $Z_{t-1}^j$
exceeds a later position $(Z_{t-1}^{j+1}, \ldots, Z_{t-1}^m, D_{t-}^1, \ldots, D_{t-}^p)$ then the next exponential clock can only begin once the present one has finished. This follows from the dynamics of $\xi$ and corresponds to the pushing interaction in the process $\Y$. 
As a result, the evolution of $(Y_t^1, \ldots, Y_t^m)$ and $(Z_t^1, \ldots, Z_t^m)$ coincides with the evolution of 
$(\Y_t^1, \ldots, \Y_t^m)$ and $(\Z_t^1, \ldots, \Z_t^m)$ whenever $t$ is odd.
If $t$ is even 
then the same argument holds with the role of $Y$ and $Z$ interchanged. 
 
We have shown the points of increase and decrease in $\hL$ and $\mathfrak{h}^{(L)}$ coincide as processes in time. Therefore $\hL$ and $\mathfrak{h}^{(L)}$ differ by a constant which must be zero since $\hL(u) \rightarrow 0$ and $\h^{(L)}(u) \rightarrow 0.$ Hence $\hL = \mathfrak{h}^{(L)}$. The proof is completed by taking a limit as $L \rightarrow \infty$. 
\end{proof}

\section{Proof of Convergence to PNG}
\label{sec:proof_thm}

\begin{proof}
The polynuclear growth model is driven by a Poisson point processes of rate $2$ on $[0, 1] \times \mathbb{R}$.
Denote this Poisson point process by $M$.
For each point $(s, x) \in M$ consider $(s^{[n]}, x)$ where $s^{[n]} = \inf\{j/2n : j/2n \geq s, j \in \mathbb{Z}_{\geq 1}\}$. 
Then $M^{[n]} = \{ (s^{[n]}, x) : (s, x) \in M\}$ 
is a collection of $2n$ independent Poisson point processes of rate $1/n$ on $\{\frac{j}{2n}\} \times \mathbb{R}$ for $j = 1, 2, \ldots, 2n$. 
We use $\{ (2n s^{[n]}, x) : (s^{[n]}, x) \in M^{[n]}\}$ as the nucleation points for the growth process in Section \ref{growth_defn2} or equivalently the up-jumps
of the auxiliary process $\xi$
in Section \ref{growth_defn}. 

Set $v_n = 1/n$. Consider the rescaled height function $(h_{\floor{2n t}})_{t \in [0, 1]}$ and 
the associated processes $(Y_{\floor{2n t}}, Z_{\floor{2n t}})_{t \in [0, 1]}.$
(Note it is possible to take the previous limits $L \rightarrow \infty$ uniformly in $n, v_n = 1/n$.)
Consider the evolution of a point of decrease $(R_t^n)_{t \geq s^{[n]}}$ emanating from a nucleation point $(s^{[n]}, x)$. 
Let $R_t = x + t - s$ for all $t \geq s$. 
We extend both processes to $0 \leq t \leq 1$ by assigning $(R_t^n)_{t \geq s^{[n]}}$ and $(R_t)_{t \geq s}$ to be in a cemetery state 
for $0 \leq t < s^{[n]}$ and $0 \leq t < s$.
In the absence of interactions $(R_t^n)_{t \geq 0}$ converges to  $(R_t)_{t \geq 0}$ in the Skorokhod topology.
This holds since the evolution of $(R_t^n)_{t \geq 0}$ is governed by an independent collection $(\zeta_k)_{k \geq 1}$ of exponential random variables with rate $v^{-1}_n = n$. It is convenient to rescale to $\zeta^*_k:= n \zeta_k$ which are exponential with rate $1$. 
For all $t \geq s^{[n]}$, by the strong law of large numbers,
\begin{align*}
R_t^n & = x+ \frac{1}{n} \sum_{j = 2n s^{[n]}: j \in 2\mathbb{Z}+1}^{\floor{2n t}} X_j^* \\
& \rightarrow x + t - s
\end{align*}
as $n \rightarrow \infty$ almost surely.

There are two types of interactions to consider: the pushing interactions between particles of the same type and annihilating interactions between particles of different types.

For all $i \geq 1$ the event that $Y_t^i$ pushes $Y_t^{i+1}$ is a large deviation event for the pair of random walks. The probability of this event decays exponentially in $n$ and by the Borel Cantelli Lemma, for sufficiently large $n$
there are almost surely no pushing interactions in the paths $(Y_{\floor{2nt}})_{t \in [0, 1]}$ and $(Z_{\floor{2nt}})_{t \in [0, 1]}$ restricted to any compact set.  

Consider two nucleation points in M given by $(s, x)$ and $(t, y)$ with $x \leq y$ where the deterministic right
and left path emanating from $(s, x)$ and $(t, y)$ respectively annihilate each other in PNG. 
In PNG these paths meet 
at time $t^* = (y + x + t -s)/2$. 
Let $\tau = \inf\{u \geq s^{[n]} \vee t^{[n]}: R_u \geq L_u\}$
Denote the processes $(R_u^n)_{s^{[n]} \leq u \leq \tau}$ and $(L_u^n)_{t^{[n]} \leq u \leq \tau}$ emanating from the corresponding nucleation points in $M^{[n]}$. 
These can be extended to processes $(R_u^n)_{0 \leq u \leq 1}$ and $(L_u^n)_{0 \leq u \leq 1}$
that are in cemetery states for $0 \leq u < s^{[n]}$, $0 \leq u < t^{[n]}$ and $u > \tau$. 
By a similar large deviation analysis to that used above we can observe that, restricting to compact set, no other particles interact with these processes
for sufficiently large $n$ almost surely. 
Let $R_u = x + u - s$ for $s \leq u < t^*$ and $L_u = y - u + t$ for $t \leq u < t^*$ and otherwise assign these processes to cemetery states.  
By the strong law of large numbers we have that almost surely $\tau \rightarrow (y + x + t -s)/2$
and that
$(R_u^n)_{0 \leq u \leq 1} \rightarrow (R_u)_{0 \leq u \leq 1}$ and 
$(L_u^n)_{0 \leq u \leq 1} \rightarrow (L_u)_{0 \leq u \leq 1}$ almost surely in the Skorokhod topology.

Therefore $(Y_{\floor{2nt}})_{t \in [0, 1]}$ and $(Z_{\floor{2nt}})_{t \in [0, 1]}$ converge in the Skorokhod topology to deterministic lines emanating from the limit of nucleation points and spreading at rate $1$ to the left and right respectively. 
Recall that for a fixed time the state space of the particle systems is the set of locally finite point measures with the topology of vague convergence. 
The height function $(h^{v_n}_{\floor{2nt}})_{t \in [0, 1]}$ is a function of the point process $M^{[n]}$ and the paths of
$(Y_{\floor{2nt}})_{t \in [0, 1]}$ and $(Z_{\floor{2nt}})_{t \in [0, 1]}$. This function has discontinuities (in the stated topologies) only when a point in $M^{[n]}$ 
coincides with a path in $(Y_{\floor{2nt}})_{t \in [0, 1]}$ and $(Z_{\floor{2nt}})_{t \in [0, 1]}$. This is given zero measure by PNG. 
The stated convergence follows from the continuous mapping theorem. 
\end{proof}

\section{Identities with point-to-line last passage percolation}
\label{sec:lpp_identities}

The growth process defined in Section \ref{growth_defn} appears at first sight to be difficult to study due to 
the fact that time is being reversed at each level. However, it can be embedded within a larger array of particles 
experiencing pushing and blocking interactions. The interactions are more complicated but the larger array is given by a
continuous-time Markov process with no need to alternately reverse time.
This process was constructed in \cite{F} and is closely related to a Brownian version in \cite{FW}. 

Let $\mathbf{e}_{ij}$ denote the vector with $1$ in the $(i, j)$-th position that is $0$ otherwise. Let $S = \{(i, j):  i, j \in \mathbb{Z}_{\geq 1}, i + j \leq 2n+1\}$ and let
\begin{equation*}
\mathcal{X} = \{ (x_{ij})_{(i, j ) \in S} : x_{ij} \in \mathbb{Z}_{\geq 0}, x_{i+1, j} \leq x_{ij} \text{ and } x_{i, j+1} \leq x_{ij} \}.
\end{equation*} 
We define a continuous-time Markov process $(X_{ij}(u) : i+j \leq 2n+1, u \in \mathbb{R})$ taking values in $\mathcal{X}$
with the following transition rates.

Suppose that $k \leq j$ and $(i, j), (i, j+1), \ldots, (i, k) \in S$. 
For $\mathbf{x}, \mathbf{x} + \mathbf{e}_{ij} +  \mathbf{e}_{i j-1} + \ldots + \mathbf{e}_{ik} \in \mathcal{X}$ and $x_{ij} = x_{i j-1} = \ldots = x_{ik}$ define
\begin{align*}
 q(\mathbf{x}, \mathbf{x} + \mathbf{e}_{ij} + \mathbf{e}_{i j-1} + \ldots + \mathbf{e}_{ik})  =  v^{1_{\{x_{ij} < x_{i-1, j+1}\}}- 1_{\{x_{ij} \geq x_{i-1, j+1}\}}} 
\end{align*}
We use the notation that $x_{0, j} = \infty$ for $j = 2, \ldots, 2n+1$.
Suppose that $l \geq i$ and $(i, j), (i+1, j), \ldots, (l, j) \in S$. 
For 
$\mathbf{x}, \mathbf{x} - \mathbf{e}_{ij} -  \mathbf{e}_{i+1 j} - \ldots - \mathbf{e}_{lj} \in \mathcal{X}$
and $x_{ij} = x_{i+1 j} = \ldots = x_{lj}$
define
\begin{align*}
 q(\mathbf{x}, \mathbf{x} - \mathbf{e}_{ij} -  \mathbf{e}_{i+1 j} - \ldots - \mathbf{e}_{lj})
  = 
v^{1_{\{x_{ij} > x_{i - 1, j+1}\}} - 1_{\{x_{ij} \leq x_{i - 1, j+1}\}}}.
\end{align*}
The transition rates describe each co-ordinate evolving as independent continuous-time random walks
with some one-sided interactions involving pushing and blocking that constrain the Markov process to remain within $\mathcal{X}$. 
Figure \ref{fig_Xarray} displays the one-sided interactions; the particle at the head of the arrow is affected by the particle at the base of the arrow.  
\begin{itemize}
\item $A \rightarrow B$ corresponds to pushing, if $A = B$ and $A$ increases by one then $B$ also increases by one, 
and blocking, if $A = B$ then $B$ cannot decrease. 
\item $\overset{A}{\underset{B}{\downarrow}}$ corresponds to pushing, if $A = B$ and $A$ decreases by one then $B$ also decreases by one, and blocking, if $A = B$ then $B$ cannot increase. 
 \item $A \leadsto B$ means that the transition rates experienced by $B$ depend on its location relative to $A$. 
 \end{itemize}
Note that the pushing interactions mean that a single jump may propagate to several particles. 
 The particles remain within $\mathbb{Z}_{\geq 0}$ which is depicted by the diagonal 
 line on the left side of Figure \ref{fig_Xarray}.
We will show that the stationary distribution is given by the probability mass function of $(G(i, j):i+j \leq 2n+1)$ 
given by
\[
\pi(x) = (1-v^2)^{n(2n+1)}v^{2\sum_{i+j < 2n+1} (x_{ij}-\max(x_{i+1, j}, x_{i, j+1})) + 2\sum_{i=1}^{2n} x_{i, 2n-i+1}}.
\]
In general, the transition rates of time reversals of Markov processes are complicated. However, in this particular case 
they are closely related to the transition rates in forwards time; all of the interactions between pairs of particles in Figure \ref{fig_Xarray} reverse their direction.  
Suppose that $k \leq i$ and $(i, j), (i, j-1), (i, k) \in S$. 
For $\mathbf{x}, \mathbf{x} + \mathbf{e}_{ij} + \mathbf{e}_{i j -1} + \ldots + \mathbf{e}_{ik} \in \mathcal{X}$ 
and $x_{ij} = x_{i j - 1} = \ldots = x_{ik}$ define
\begin{equation*}
\hat{q}( \mathbf{x} + \mathbf{e}_{ij} + \mathbf{e}_{i j -1} + \ldots + \mathbf{e}_{ik}, \mathbf{x}) = 
 v^{1_{\{ x_{ik} \geq x_{i+1, k-1}\}} - 1_{\{ x_{ik} < x_{i+1, k-1}\}}}.
\end{equation*} 
We use the notation that $x_{j0} = \infty$ for $j = 2, \ldots, 2n+1$.
Suppose $l \geq i$ and $(i, j), (i+1, j), \ldots, (l, j) \in S$.
For $\mathbf{x}, \mathbf{x} -\mathbf{e}_{ij} - \mathbf{e}_{i+1, j} - \ldots - \mathbf{e}_{lj}$
and $x_{ij} = x_{i+1 j} = \ldots = x_{lj}$ define 
\begin{equation*}
\hat{q}(\mathbf{x} -\mathbf{e}_{ij} - \mathbf{e}_{i+1, j} - \ldots - \mathbf{e}_{lj}, \mathbf{x}) = 
 v^{1_{\{x_{lj} \leq x_{l+1, j-1}\}}-1_{\{x_{lj} > x_{l+1, j-1} \}}}.
 \end{equation*}
It was shown in \cite{F} that 
\begin{align*}
\pi(x) q(x, x') &= \pi(x') \hat{q}(x', x), \qquad x, x' \in \mathcal{X} \\
\sum_{x' \neq x} q(x, x') & = \sum_{x' \neq x} \hat{q}(x, x') \qquad x \in \mathcal{X}. 
\end{align*}
Therefore $\pi$ is the stationary distribution and $\hat{q}$ are the transition rates in reversed time 
when the process is run in stationarity. 
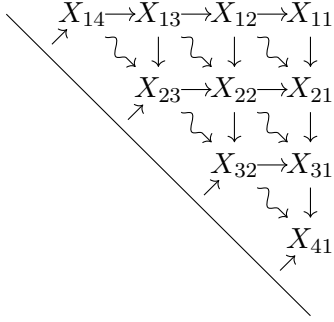
\begin{figure}
\centering
 \begin{tikzpicture}
 \node at (4, 4) {$X_{11}$};
 \draw[->] (3.3, 4) -- (3.7, 4);
  \draw[->] (4, 3.7) -- (4, 3.3);
 \node at (3, 4) {$X_{12}$};
 \draw[->] (2.3, 4) -- (2.7, 4);
   \draw[->] (4, 2.7) -- (4, 2.3);
 \node at (2, 4) {$X_{13}$};
 \draw[->] (1.3, 4) -- (1.7, 4);
   \draw[->] (4, 1.7) -- (4, 1.3);
 \node at (1, 4) {$X_{14}$};
 \draw[->] (3.3, 3) -- (3.7, 3);
   \draw[->] (3, 3.7) -- (3, 3.3);
 \node at (4, 3) {$X_{21}$};
 \draw[->] (2.3, 3) -- (2.7, 3);
   \draw[->] (3, 2.7) -- (3, 2.3);
 \node at (3, 3) {$X_{22}$};
 \draw[->] (3.3, 2) -- (3.7, 2);
 \node at (2, 3) {$X_{23}$};
   \draw[->] (2, 3.7) -- (2, 3.3);
 \node at (4, 2) {$X_{31}$};
 \node at (3, 2) {$X_{32}$};
 \node at (4, 1) {$X_{41}$};
 \draw (0, 4) -- (4, 0);
 \path[draw = black, ->, snake it]    (3.3, 1.7) -- (3.7,1.3);
 \path[draw = black, ->, snake it]    (2.3, 2.7) --  (2.7,2.3);
 \path[draw = black, ->, snake it]    (1.3,3.7) -- (1.7, 3.3);
 \path[draw = black, ->, snake it]    (3.3,2.7) -- (3.7, 2.3);
 \path[draw = black, ->, snake it]    (3.3,3.7) -- (3.7, 3.3);
 \path[draw = black, ->, snake it]    (2.3,3.7) -- (2.7, 3.3);
  \draw[->] (0.6, 3.6) -- (0.8, 3.8);
  \draw[->] (1.6, 2.6) -- (1.8, 2.8);
   \draw[->] (2.6, 1.6) -- (2.8, 1.8);
    \draw[->] (3.6, 0.6) -- (3.8, 0.8);
 \end{tikzpicture}
\caption{The interactions in the system $\{X_{ij} : i + j \leq n+1\}$.}
\label{fig_Xarray}
\end{figure}
This is Theorem 5.2 in \cite{F}, for any $n \geq 1$ the invariant measure of $(X_{ij}(u) : i+j \leq 2n+1, t \geq 0)$ is 
equal in distribution to $(G(i, j) : i+j \leq 2n+1)$.
Furthermore, when run in stationarity, 
\begin{equation*}
(X_{ij}(u))_{u \in \mathbb{R}, i+j \leq 2n+1} \stackrel{d}{=} (X_{ji}(-u))_{u \in \mathbb{R}, i+j \leq 2n+1}. 
\end{equation*}
In our context this property corresponds to spatial symmetry of the growth process in Section \ref{growth_defn}.
The final property that we require from \cite{F} concerns the marginal distributions of rows and columns. 
In particular, the marginal distribution of any row $(X_{i, 2n-i+1}, \ldots, X_{i, 1})$ run forwards in time is PushASEP with a wall at the origin. PushASEP is an interacting particle system introduced in \cite{BF} where each co-ordinate evolves independently according to a continuous-time random walk with rate $v$ of jumping to the right, rate $v^{-1}$ of jumping to the left and one-sided interactions involving pushing and blocking
that preserve the ordering $X_{i, 2n-i+1} \leq \ldots \leq X_{i, 1}.$ With the meaning of arrows above, and an additional wall at the origin, 
the interactions can be depicted as
\[
\vert \rightarrow X_{i, 2n-i+1} \rightarrow \ldots \rightarrow X_{i, 1}.
\]
Moreover, the marginal distribution of a column $(X_{2n-j+1, j}, \ldots, X_{1, j})$ run backwards in time is PushASEP with 
a wall at the origin.

We now show that the growth process in Section \ref{growth_defn} run up to time $2n$, in a stationary regime $L \rightarrow \infty$, can be found 
as a marginal of $(X_{ij}(u) : i+j \leq 2n+1, u \in \mathbb{R})$.
\begin{proposition}
\label{prop:height_array}
When $X$ is run in stationarity
\begin{align*}
& (X_{n+1, n}(u), X_{n, n}(u), X_{n, n-1}(u), \ldots, X_{11}(u))_{u \in \mathbb{R}} \\
& \stackrel{d}{=} (h_{1}(u), h_{2}(u), \ldots, h_{2n}(u))_{u \in \mathbb{R}}.
\end{align*}
\end{proposition}

\begin{proof}
Consider $(X_{ij} : (i, j) \in S)$ run backwards in time. 
The directions of interactions backwards in time show that
for any $1 \leq k \leq n$,
 \[(X_{n +1, n}, X_{n, n}, X_{n, n -1}, \ldots, X_{kk})\]
is conditionally independent of $X_{k-1, k-1}$
given \[
(X_{k, k-1}, X_{k+1, k-1}, \ldots, X_{2n-k+2, k-1}).
\]
It also follows from the direction of interactions backwards in time that
$X_{k-1, k-1}$ is conditionally independent of $(X_{k+1, k-1}, \ldots, X_{2n-k+2, k-1})$ given $X_{k, k-1}$. 
This can be combined to show that 
for any $1 \leq k \leq n$, 
\[(X_{n +1, n}, X_{n, n}, X_{n, n -1}, \ldots, X_{kk})\]
is conditionally independent of $X_{k-1, k-1}$ given $X_{k, k-1}$.

A similar argument, using the directions of interactions forwards in time, shows that  
$X_{k, k-1}$ is conditionally independent of 
\[(X_{n+1, n}, X_{n, n}, X_{n, n-1} \ldots, X_{k+1, k})\] given $X_{k, k}$.
Hence let $\mathscr{H}_0 \equiv 0$ and 
\begin{align*}
\mathscr{H}_{2j+1} & = X_{n + 1-j, n - j}, \qquad j = 0, \ldots, n-1 \\
\mathscr{H}_{2j} & = X_{n + 1-j, n +1- j}, \qquad j = 1, \ldots, n.
\end{align*}
Then $(\mathscr{H}_j)_{j = 0}^{2n}$ is a Markov process and by taking modifications its state space can be taken to be $\UC$ as in Section \ref{growth_defn}. 

The distribution of 
$X_{k, k-1}$ given $X_{k, k}$ follows from the fact that the marginal distribution of 
\[
(X_{k, 2n-k+1}, \dots, X_{k, k}, X_{k, k-1}, \ldots, X_{k, 1})
\]
is PushASEP with a wall at the origin. 
This coincides with the dynamics of $h^{(L)}$ given in Section \ref{growth_defn} except that it is has been defined directly in a stationary regime.
The distribution of 
$X_{k-1, k-1}$ given $X_{k, k-1}$ follows in a similar manner and again coincides with the definition of $h^{(L)}$ in a stationary regime. As $L \rightarrow \infty$, $h^{(L)}$ converges to $h$ which satisfies the stated identity. 
\end{proof}

\begin{proof}[Proof of Theorem \ref{main_thm} part (ii)]
This is a consequence of Propostion \ref{prop:height_array} and 
that the invariant measure of $(X_{ij}(u) : i+j \leq 2n+1, u \geq 0)$ is 
equal in distribution to $(G(i, j) : i+j \leq 2n+1)$.
\end{proof}

Theorem \ref{main_thm} part (i) can be extended to limits of different marginals of $(X_{ij}(u))_{u \in \mathbb{R}, i+j \leq 2n+1}$ as $v_n = 1/n$ and $n \rightarrow \infty$. 
A particularly interesting case is the top row that has the marginal distribution of PushASEP with a wall. This satisfies
\begin{equation}
\label{push_asep_limit}
(X_{1, 2n+1-\floor{2nt}}^{v_n})_{t \in [0, 1]} \rightarrow (\hPNG_t(\cdot -t))_{t \in [0, 1]}
\end{equation} 
in the sense of weak convergence on $\D$ as 
$n \rightarrow \infty$. The proof of this follows the proof of Theorem \ref{main_thm} part (i) very closely. 
The difference is that time is not being reversed at alternative levels so the points of increase of the interface remain fixed
while the points of decrease move during every time step. In the limit, this corresponds to the left boundaries of islands in flat PNG remaining fixed while the right boundaries move to the right at speed 2. This can be viewed as flat PNG in a moving frame of reference. 

A final question is to take a limit $v_n = 1/n, n \rightarrow \infty$ of the whole array $(X_{ij}(u))_{u \in \mathbb{R}, i+j \leq 2n+1}$. 
 Comparing Theorem \ref{main_thm} part (i) and Equation \eqref{push_asep_limit} 
shows that two particular marginals converge to PNG in different moving frames of reference. My belief is that the whole array converges to PNG, simultaneously viewed in different frames of references. This is omitted since the limit is degenerate.

\bibliographystyle{abbrv}
 \bibliography{growth_model}

\end{document}